\documentclass[11pt]{amsart}
\usepackage {amsmath, amssymb,   epsfig, a4wide, enumerate, psfrag}
\usepackage[utf8]{inputenc}
\usepackage[english]{babel}

\date{\today}

\author{Romain Dujardin}
\thanks{Research  partially supported by ANR project LAMBDA,  ANR-13-BS01-0002 and  a grant from the  Institut Universitaire de France}
 
\title{A closing lemma for polynomial automorphisms of $\mathbb{C}^2$}

\address{Sorbonne Universités, UPMC Univ Paris 06, Laboratoire de Probabilités et Modèles Aléatoires (LPMA, UMR 7599),   4 place Jussieu, 75252 Paris Cedex 05, France}
\email{romain.dujardin@upmc.fr}


\newcommand{\cc}{\mathbb{C}}

\newcommand{\dd}{\mathbb{D}}
\newcommand{\zz}{\mathbb{Z}}

\newcommand{\e}{\varepsilon}

 \newcommand{\cv}{\rightarrow}

\newcommand{\fr}{\partial}
\newcommand{\om}{\Omega}
\newcommand{\set}[1]{\left\{#1\right\}}
\newcommand{\norm}[1]{\left\Vert#1\right\Vert}
\newcommand{\abs}[1]{\left\vert#1\right\vert}
\newcommand{\cd}{{\cc^2}}

\newcommand{\rest}[1]{ \arrowvert_{#1}}

\newcommand{\unsur}[1]{\frac{1}{#1}}

\newcommand{\la}{\lambda}

\newcommand{\hot}{ {h.o.t.}}
\newcommand{\loc}{\mathrm{loc}}

\newcommand{\inv}{^{-1}}

\DeclareMathOperator{\supp}{Supp}

\DeclareMathOperator{\jac}{Jac}

\newtheorem{prop}{Proposition} [section]
\newtheorem{thm}[prop] {Theorem}

\newtheorem{lem}[prop] {Lemma}
\newtheorem{cor}[prop]{Corollary}

\theoremstyle{remark}

\begin{document}

\begin{abstract}
We prove that for a polynomial diffeomorphism of $\cd$,
the support of any invariant measure, apart from  a few obvious cases,  is contained in the closure of the 
set of saddle periodic points. 
\end{abstract}

 \maketitle
 
 \section{Introduction and results}
 
 Let $f$ be a polynomial diffeomorphism of $\cd$  with non-trivial dynamics. This hypothesis  can
 be expressed in a variety of ways, for instance  it is equivalent to the positivity of  topological entropy. 
The dynamics of such transformations has attracted a lot of attention in the past few decades  (the reader can consult 
 e.g.  \cite{bedford} for   basic facts and references). 
 
  In this paper we make the standing assumption that $f$ is dissipative, i.e. 
 that the (constant) Jacobian of $f$ satisfies $\abs{\jac(f)}<1$.
 
 We classically denote by $J^+$   the forward   Julia set,
 which can be characterized as usual in terms of normal families, or by saying that  
 $J^+ = \fr K^+$, where $K^+$ is the set of points with bounded forward orbits. Reasoning analogously for backward iteration 
 gives the backward Julia set $J^- = \fr K^-$.  
 Thus the 2-sided Julia set is naturally defined by $J = J^+\cap J^-$
Another interesting dynamically defined subset  is    
the closure  $J^*$   of the set of saddle periodic points (which is also the support of the unique entropy maximizing 
 measure \cite{bls}). 
 
 The inclusion $J^*\subset J$ is obvious. 
 It is a major open question in this area of research    whether the converse inclusion holds.
Partial answers have been given in \cite{bs1, bs3, connex, lyubich peters 2, peters guerini}.
 
 \medskip

 Let  $\nu$ be an ergodic $f$-invariant   probability measure. If 
  $\nu$ is hyperbolic,  that is, its two    Lyapunov exponents\footnote{Recall that in holomorphic dynamics, 
  Lyapunov exponents always have even multiplicity.} are non-zero and of  opposite sign, 
 then the so-called Katok closing lemma \cite{katok} implies that $\supp(\nu) \subset J^*$. It may also be the case that $\nu$ is supported in the Fatou set: then from the classification of recurrent Fatou components in \cite{bs2}, this happens if and only if $\nu$ is supported on an attracting or  semi-Siegel periodic orbit, or is the Haar measure on
  a cycle  of $k$ circles along which $f^k$ is conjugate to an irrational  rotation (recall that 
 $f$ is assumed dissipative). Here by semi-Siegel periodic orbit, 
 we mean   a linearizable periodic orbit with one attracting   and one 
 irrationally indifferent multipliers. 
 
  The following ``ergodic closing lemma'' is the main result of this note:
 
 \begin{thm}\label{thm:main}
 Let $f$ be a dissipative polynomial diffeomorphism of $\cd$ with non-trivial dynamics, and $\nu$ be any invariant measure supported on $J$. Then $\supp(\nu)$ is contained in $J^*$. 
 \end{thm}
 
 A consequence is that if $J\setminus J^*$ happens to be
  non-empty,  then the dynamics on $J\setminus J^*$ is ``transient" in a measure-theoretic sense. 
  Indeed, if $x\in J$, we can form an invariant 
 probability measure by taking a cluster limit  of $\unsur{n}\sum_{k=0}^n \delta_{f^k(x)}$ 
 and the theorem says that any such invariant measure will be concentrated on $J^*$. More generally the same argument implies:

\begin{cor}
Under the assumptions of the theorem, if $x\in J^+$, then $\omega(x)\cap J^*\neq \emptyset$. 
\end{cor}

Here as usual $\omega(x)$ denotes the $\omega$-limit set of $x$.  Note that for  $x\in J^+$ then it is obvious 
that $\omega(x)\subset J$. 
It would be interesting to know whether  the conclusion of the corollary can be replaced by the sharper one: $\omega(x)\subset J^*$. 
 
 \medskip
 
 Theorem \ref{thm:main} can be formulated slightly more precisely as follows.
 
 \begin{thm}\label{thm:precised}
 Let $f$ be a dissipative polynomial diffeomorphism of $\cd$ with non-trivial dynamics, and $\nu$ be any ergodic 
 invariant probability measure. Then one of the following situations holds:
 \begin{enumerate}[{(i)}]
  \item either $\nu$ is atomic and supported on an attracting or semi-Siegel cycle;
 \item or $\nu$ is the Haar measure on an invariant cycle of circles contained in a  periodic rotation domain;
 \item or $\supp(\nu)\subset J^*$.
 \end{enumerate}
 \end{thm}

 Note that the additional ergodicity  assumption on $\nu$ is harmless since any invariant measure is an integral of ergodic ones. 
 The only new ingredient with respect to Theorem \ref{thm:main}
  is the fact that measures supported on periodic orbits that do not fall in case {\em (i)}, that is, 
 are either semi-parabolic or semi-Cremer, are supported on $J^*$. For semi-parabolic points this is certainly known to the experts although  apparently not available in print.
 For semi-Cremer points this follows from the hedgehog construction of Firsova, Lyubich, Radu and Tanase 
  (see   \cite{lyubich radu tanase}). For completeness we give complete proofs below. 
 
 \medskip
 
 \noindent{\bf Acknowledgments.} Thanks to Sylvain Crovisier and Misha Lyubich for inspiring conversations.
  This work was motivated
  by the work of Crovisier and Pujals  on strongly 
 dissipative diffeomorphisms  (see \cite[Thm 4]{crovisier pujals}) and by the work of Firsova, Lyubich, Radu and Tanase 
 \cite{firsova lyubich radu tanase, lyubich radu tanase}
 on hedgehogs in higher dimensions (and the question whether  hedgehogs for   Hénon maps are contained in $J^*$).

 \section{Proofs}
 
In this section we prove Theorem \ref{thm:precised} by dealing separately with the atomic and the non-atomic case. Theorem \ref{thm:main} follows immediately. Recall that $f$ denotes a 
dissipative  polynomial diffeomorphism with non trivial dynamics and $\nu$   an $f$-invariant ergodic probability measure. 

  \subsection{Preliminaries}

Using the theory of laminar currents, it was shown in \cite{bls} that any saddle periodic point belongs to $J^*$. More generally, if $p$ and $q$ are saddle points, then 
$J^* = \overline{ {W^u(p)\cap W^u(q)}}$ (see Theorems 9.6 and 9.9 in \cite{bls}). 
This result was generalized in  \cite{tangencies} as follows.  If $p$ is any saddle point and $X\subset W^u(p)$, we respectively 
denote by $\mathrm{Int}_i X$, $\mathrm{cl}_i X$, $\fr_i X$ the interior, closure and boundary  of 
$X$ relative to the intrinsic topology of
$W^u(p)$, that is the topology induced by the biholomorphism $W^u(p)\simeq \cc$.   
 
 \begin{lem}[{\cite[Lemma 5.1]{tangencies}}]\label{lem:homoclinic}
 Let $p$ be a saddle periodic point.
 Relative to  the intrinsic topology in $W^u(p)$, $\fr_i(W^u(p)\cap K^+)$ is 
 contained in the closure of the set of transverse homoclinic intersections. In particular $\fr_i(W^u(p)\cap K^+)\subset J^*$. 
 \end{lem}
 
 Here is another statement along the same lines, which can easily be extracted from \cite{bls}. 
 
 \begin{lem}\label{lem:entire}
 Let $\psi:\cc\to \cd$ be an entire curve such that $\psi(\cc)\subset K^+$. Then for any saddle point 
 $p$, $\psi(\cc)$ admits transverse intersections with $W^u(p)$.
 \end{lem}

\begin{proof}
This is identical to  the first half of the  proof  of \cite[Lemma 5.4]{tangencies}.
\end{proof}

We will repeatedly use the following alternative which follows from the combination of the two previous lemmas.  Recall that 
a Fatou disk is a holomorphic disk along which the iterates $(f^n)_{n\geq 0}$ form a normal family. 

\begin{lem}\label{lem:inters}
Let $\mathcal{E}$ be an entire curve contained in $K^+$, $p$ be any saddle point, and $t$ be a transverse intersection point between 
$\mathcal{E}$ and $W^u(p)$. Then either $t\in J^*$ or there is a Fatou disk $\Delta\subset W^u(p)$ containing $t$.  
\end{lem}

\begin{proof}
Indeed, either $t\in \fr_i(W^u(p)\cap K^+)$ so by Lemma \ref{lem:homoclinic}, $t\in J^*$, 
or $t\in \mathrm{Int}_i(W^u(p)\cap K^+)$. In the latter case,
 pick any open disk $\Delta\subset\mathrm{Int}_i(W^u(p)\cap K^+) $ containing $t$. Since $\Delta$ is contained in  
$K^+$, its forward iterates remain bounded so it is a Fatou disk. 
\end{proof}

\subsection{The atomic case}
Here we prove Theorem \ref{thm:precised} when $\nu$ is   atomic. By ergodicity, this implies that $\nu$ is concentrated on a single
periodic orbit.
  Replacing $f$ by an iterate we may assume that it is concentrated on a fixed point. Since $f$ is dissipative there must be an attracting eigenvalue. A first possibility is that this fixed point is attracting or semi-Siegel. Then we are 
in case {\em (i)} and there is nothing to say. Otherwise 
$p$ is semi-parabolic or semi-Cremer and we must show that $p\in J^*$. In both cases, $p$ admits a strong stable manifold 
$W^{ss}(p)$ associated to the contracting eigenvalue, which is biholomorphic to $\cc$  by a theorem of Poincaré. Let $q$ be a 
saddle periodic
point and $t$ be a point of   transverse intersection   between $W^{ss}(p)$ and $W^u(q)$. If $t\in J^*$, then 
since $f^n(t)$  converges to $p
$ as $n\to \infty$ we are done. Otherwise there is a non-trivial Fatou disk $\Delta$ transverse to  $W^{ss}(p)$ at $t$. 
Let us   show  that this is contradictory. 

 In the semi-parabolic case, this is classical. A short argument goes as follows  (compare \cite[Prop. 7.2]{ueda}). Replace $f$ by an iterate so that the neutral eigenvalue is equal to 1.
  Since $f$ has no curve of fixed points  there are local coordinates $(x,y)$
near $p$ in which $p=(0,0)$, $W^{ss}_{\rm loc}(p)$ is the $y$-axis $\set{x=0}$
 and $f$ takes the form 
$$(x,y) \longmapsto   (x+ x^{k+1}+  \hot , by+  \hot)\ , $$ with $\abs{b}<1$
(see \cite[\S 6]{ueda}). Then $f^n $ is of the form 
 $$(x,y) \longmapsto   (x+ n x^{k+1}+  \hot , b^n y +  \hot  ) \ , $$
so we see that $f^n$ cannot be normal along any disk transverse to the $y$ axis and we are done. 

 In the semi-Cremer case we rely on the hedgehog theory of \cite{firsova lyubich radu tanase, 
 lyubich radu tanase}.  
 Let $\phi:\dd\to \Delta$ be any parameterization, and fix local coordinates $(x,y)$ as before in which
 $p=(0,0)$, $W^{ss}_{\rm loc}(p)$ is the $y$-axis  
 and $f$ takes the form $$(x,y) \longmapsto   (e^{i2\pi\theta} x  , by) +  \hot$$ 
 Let $B$ be a small neighborhood of the origin in which the hedgehog is well-defined. 
 Reducing $\Delta$ and iterating a few times if necessary, we can 
 assume that for all $k\geq 0$, $f^k(\Delta)\subset B$ and $\phi$ is of the form $s\mapsto   (s, \phi_2(s))$. 
Then the first coordinate of 
$f^n\circ \phi $ is of the form $s\mapsto e^{i2n\pi\theta} s+  \hot$. 
If  $(n_j)_{j\geq 0}$ is a 
subsequence  such that $f^{n_j}\circ \phi$ converges to some $\psi = (\psi_1, \psi_2)$, we get that
 $\psi_1(s)   =  \alpha s+ h.o.t.$, 
where $\abs{\alpha} = 1$. Thus $\psi(\dd)  = \lim f^{n_j}(\Delta)$ is a non-trivial holomorphic 
disk $\Gamma$ through 0 that is smooth at the origin. 
  
%
 For 
every $k\in \zz$ we have that $f^{k}(\Gamma) = \lim f^{n_j+k}(\Delta)\subset B$. 
Therefore by the local uniqueness of hedgehogs (see \cite[Thm 2.2]{lyubich radu tanase})
$\Gamma$ is contained in $\mathcal{H}$. It follows that  $\mathcal{H}$ has non-empty relative interior in any local center 
manifold of $p$ and 
 from \cite[Cor. D.1]{lyubich radu tanase} we infer  that $p$ is semi-Siegel, which is the desired contradiction. 
 
 \subsection{The non-atomic case}
 Assume now that $\nu$ is non-atomic. If $\nu$ gives positive mass to the Fatou set, then by  
   ergodicity it must give full mass to a cycle of  recurrent Fatou components.  These were classified in 
 \cite[\S 5]{bs2}: they are  either attracting basins or rotation domains.    
 Since $\nu$ is non-atomic we must be in the second situation. Replacing $f$ by $f^k$ we may assume that we are
  in a fixed 
 Fatou component $\om$. Then $\om$ retracts onto some Riemann surface $S$ which is a biholomorphic 
 to a disk or an annulus  and  on which the 
 dynamics is that of an   irrational rotation. Furthermore all orbits in $\om$ converge to $S$. Thus $\nu$ must give full mass to $S$, and since $S$ is foliated by invariant circles, by ergodicity $\nu$   gives full mass to a single circle. Finally the 
 unique ergodicity of irrational rotations implies that $\nu$ is the Haar measure. 
 
 \medskip
 
 Therefore we are left with the case where $\supp(\nu)\subset J$, that is, we must prove Theorem \ref{thm:main}. 
 Let us  start by recalling some   facts on the  Oseledets-Pesin theory of our mappings. 
 Since $\nu$ is ergodic by the Oseledets theorem there exists 
 $1\leq k\leq 2$,   a set $\mathcal R$ of full measure and for $x \in \mathcal R$ a 
 measurable splitting of $T_x\cd$, $T_x\cd = \bigoplus_{i=1}^k E_i(x)$ such that 
 for $v\in E_i(x)$, $\lim_{n\cv\infty}\unsur{n}\log  \norm{df^n_x(v)}   = \chi_i$.    
 Moreover, $\sum \chi_i   = \log \abs{\jac(f)}<0$, 
 and since $\nu$ is non-atomic both $\chi_i$ cannot be both negative (this is already part of Pesin's theory, see \cite[Prop. 2.3]{bls}). 
  Thus $k=2$  and  the exponents satisfy   $\chi_1<0$ and $\chi_2\geq 0$ (up to relabelling). 
Without  loss of generality, we may further assume that points in $\mathcal R$ satisfy 
   the conclusion of the Birkhoff ergodic theorem for $\nu$.

  As observed in the introduction,
  the ergodic closing lemma is well-known when $\chi_2 >0$ so
   we might only consider the case $\chi_2=0$ (our proof actually  treats both cases simultaneously). 
    
 To ease notation, let us denote by $E^s(x)$ the stable Oseledets subspace and by $\chi^s$ the corresponding Lyapunov exponent
 ($\chi^s<0$). The Pesin stable manifold theorem (see e.g. \cite{fathi herman yoccoz} for  details)
 asserts that 
  there exists a measurable 
 set $\mathcal{R}'\subset \mathcal R$ of full measure,  
 and a family  of holomorphic disks
 $W^s_\loc (x)$,
  tangent to $E^s(x)$ at $x$ for  $x\in \mathcal R'$, and such that 
  $f(W^s_\loc (x))\subset W^s_\loc (f(x))$. In addition for every $\e>0$ there exists a set 
  $\mathcal{R}'_\e$  of measure  $\nu(\mathcal{R}'_\e) \geq 1-\e$ and constants 
  $r_\e$ and $C_\e$ 
  such that for 
  $x\in \mathcal{R}'_\e$, $W^s_\loc (x)$ contains a graph of slope at most 1 over a ball of radius $r_\e$ in $E^s(x)$ 
  and for $y\in W^s_\loc (x)$, $d(f^n(y), f^n(x))\leq C_\e\exp ( (\chi^s+\e)n)$ for every $n\geq 0$. Furthermore, local stable 
  manifolds vary continuously on $\mathcal{R}'_\e$. 
 
 From this we can form global stable manifolds by declaring\footnote{If $\nu$ has a zero exponent, this may not be the stable manifold of $x$ in the usual sense, that is, there might exists points outside $W^s(s)$ whose orbit approach that of $x$.} that $W^s(x)$ is the increasing union 
 of $f^{-n} (W^s_\loc(f^n(x)))$. Then  it is a well-known fact that $W^s(x)$ is a.s. biholomorphically equivalent to $\cc$
 (see e.g. \cite[Prop 2.6]{bls}). Indeed, 
 almost every point visits    $\mathcal{R}'_\e$ infinitely many times, and from this we 
 can view $W^s(x)$ as an increasing union of disks $D_j$ such that the modulus of the annuli $D_{j+1}\setminus D_j$ is 
 uniformly bounded from below.  Discarding a set of zero measure if necessary, it is no loss of generality to 
  assume that  $\bigcup_{\e>0} \mathcal R'_\e = \mathcal R'$ and that for every 
 $x\in \mathcal{R}'$, $W^s(x)\simeq \cc$. 
  
  \medskip
  
 To prove the theorem we show that for  every $\e>0$, $\mathcal R'_\e \subset J^*$. 
 Fix  $x\in \mathcal R'_\e$ and a  saddle point $p$.  By Lemma \ref{lem:entire}  
 there is a transverse intersection $t$ between $W^s(x)$ 
 and $W^u(p)$. Since $x$ is recurrent and $d(f^n(x), f^n(t))\to 0$, 
 to prove that $x\in J^*$ it is enough to show that $t\in J^*$. We argue by contradiction so assume that this is not the case. Then by Lemma \ref{lem:inters} there is a Fatou disk $\Delta$ through 
 $t$ inside $W^u(p)$. Reducing $\Delta$ a little if necessary we may assume that $f^n$ is a normal family in some
  neighborhood of $\overline \Delta$ in $W^u (p)$. 
  
   Since $\nu$ is non-atomic and stable manifolds vary continuously for the $C^1$ topology on 
 $\mathcal R'_\e$, there is a set $A$ of positive measure such that if $y\in A$, $W^s(y)$ 
 admits a transverse intersection with $\Delta$. The iterates $f^n(\Delta)$ form a normal family and $f^n(\Delta)$ is 
 exponentially close to $f^n(A)$. Let $(n_j)$ be some subsequence such that $f^{n_j}\rest\Delta$ converges. Then the
  limit map has either generic rank 0 or 1, that is if $\phi : \dd\to \Delta$ is a parameterization, $f^{n_j}\circ \phi$ converges uniformly on $ \dd$ to some limit map $\psi$, which is either constant or has generic rank 1. 
  Set $\Gamma = \psi(\dd)$. Let $\nu'$ be a cluster 
  value of the sequence of measures $(f^{n_j})_*(\nu\rest{A})$. Then $\nu'$ is a measure of mass 
  $\nu(A)$, supported on $\overline \Gamma$ and $\nu'\leq \nu$.  Since $\nu$ gives no mass to points, the rank 0 case is excluded so $\Gamma$ is a (possibly singular) curve.  Notice also that if $z$ is an interior point of $\Delta$ 
  (i.e. $z= \phi(\zeta)$ for some $\zeta \in \dd$), then $\lim f^{n_j}(z)   = \psi(\zeta)$ is an interior point of $\Gamma$. This shows that $\nu'$ gives full mass to $\Gamma$ (i.e. it is not concentrated on its boundary).  Then 
  the proof of Theorem \ref{thm:main} 
  is concluded by the following result of independent interest. 
 
  \begin{prop}\label{prop:subvariety}
  Let $f$ be a dissipative polynomial diffeomorphism of $\cd$ with non-trivial dynamics, and $\nu$ be an ergodic non-atomic
   invariant measure, giving positive measure to a subvariety. Then  
   $\nu$ is the Haar measure on an invariant cycle of circles contained in a  periodic rotation domain.
   
    In particular  a non-atomic invariant measure supported on $J$ gives no mass to subvarieties.
  \end{prop}

 \begin{proof}
 Let $f$ and $\nu$ be as in the statement of the proposition, and $\Gamma_0$ be a subvariety such that $\nu(\Gamma_0)>0$.
 Since $\nu$ gives no mass to the singular points of $\Gamma_0$, by reducing $\Gamma_0$ a bit we may assume that $\Gamma_0$ is smooth.  If  $M$ is  an integer such that $1/M < \nu(\Gamma_0)$, by the pigeonhole principle
  there exists $0\leq k \leq l \leq M$ such that 
$\nu(f^k(\Gamma_0)\cap f^l(\Gamma_0))>0$, so $f^k(\Gamma_0)$ and $f^l(\Gamma_0)$ 
intersect along a relatively open set. Thus replacing 
$f$ by some iterate $f^N$ (which does not change the Julia set) 
we can assume that $\Gamma_0\cap f(\Gamma_0)$ is 
relatively open in $\Gamma_0$ and $f(\Gamma_0)$.  
Let now $\Gamma = \bigcup_{k\in \zz} f^k(\Gamma_0)$. This is an invariant, 
injectively 
immersed Riemann surface with $\nu(\Gamma)>0$. Notice that replacing $f$ by $f^N$ may  
 corrupt  the ergodicity 
of $\nu$ so if needed we replace $\nu$ by a component of its ergodic decomposition (under $f^N$)
giving positive (hence full) mass to $\Gamma$. 
 
 \medskip
 
 We claim that $\Gamma$ is biholomorphic to   a  domain of the form $\set{z\in \cc, \ r <\abs{z} <R}$ 
 for some $0\leq r< R \leq \infty$,  that 
 $f\rest{\Gamma_0}$ is conjugate to an irrational rotation, and $\nu$ is the Haar measure on an invariant circle. This is 
 {\em a priori} not enough to 
 conclude the proof since at this stage nothing prevents such an invariant ``annulus'' to be contained in $J$. 
 
 To prove the claim, note first that since $\Gamma$ is non-compact, it is either biholomorphic 
   to $\cc$  or   $\cc^*$, or  it is a hyperbolic Riemann surface\footnote{In the situation of Theorem \ref{thm:main} we further know that $ \Gamma\subset K$ so the first two cases are excluded.}. In addition  $\Gamma$ possesses an automorphism $f$
  with a non-atomic ergodic  invariant measure. In the case of $\cc$ and $\cc^*$ all automorphisms are affine and 
  the only possibility is that $f$ is an irrational rotation. In the case of a hyperbolic Riemann surface, the list of possible dynamical 
  systems is also well-known (see e.g. \cite[Thm 5.2]{milnor}) and again the only possibility is that $f$ is conjugate to 
  an irrational rotation on a disk or an annulus. The fact that $\nu$ is a Haar measure follows as before. 
 
 \medskip
 
Let $\gamma $ be the circle supporting $\nu$, and    
 $\widetilde \Gamma \subset \Gamma$ be a relatively compact invariant 
annulus containing $\Gamma$ in its interior. 
 To conclude the proof we must show that $\gamma$ is contained in the Fatou set. 
 This will result  from the following lemma, which will be proven afterwards.
 
 \begin{lem}\label{lem:dominated}
$f$ admits a dominated splitting along $\widetilde \Gamma$. 
 \end{lem}

See \cite{sambarino} for generalities on the notion of dominated splitting.
In our setting, since $\Gamma$ is an invariant complex submanifold and  $f$ is dissipative, 
the dominated splitting actually implies  a normal hyperbolicity property.
Indeed, observe first that 
 $f\rest{\tilde \Gamma}$ is an isometry for the Poincaré metric  $\mathrm{Poin}_\Gamma$ of $\Gamma$, 
 which is equivalent to the induced Riemannian metric on $\widetilde \Gamma$. In particular  
  $C\inv \leq \norm{df^n\rest{T\tilde\Gamma}}\leq C$ for some $C>0$ independent of $n$.  
 Therefore a dominated splitting for 
 $f\rest{\widetilde \Gamma}$  means that there is a continuous splitting of
  $T\cd$ along $\widetilde \Gamma$, $T_x\cd = T_x\Gamma\oplus V_x$, and for every 
  $x\in \widetilde \Gamma$ and   $n\geq 0$ we have $\norm{df_x^n\rest{V_x}} \leq C'\lambda^n$ for some $C'>0$ and 
   $\lambda <1$.    
In other words,  $f$ is normally contracting along $\widetilde\Gamma$.  
 Thus in a neighborhood of $\gamma$, 
all orbits  converge to $\Gamma$. This completes the proof of   Proposition \ref{prop:subvariety}.
 \end{proof}

 \begin{proof}[Proof of Lemma \ref{lem:dominated}] 
 By the cone criterion for dominated splitting (see \cite[Thm 1.2]{newhouse cone} or \cite[Prop. 3.2] {sambarino}) it is enough to 
 prove that for every $x\in \Gamma$ there exists a cone $\mathcal{C}_x$ about $T_x \Gamma$ 
 in $T_x\cd$ such that  the field of 
 cones   $(\mathcal{C}_x)_{x\in \widetilde \Gamma}$
  is strictly contracted by the dynamics. For   $x\in   \Gamma$, choose a vector $e_x\in T_x\Gamma$ of unit norm relative to the Poincaré metric 
 $\mathrm{Poin}_\Gamma$
  and  pick $f_x$ orthogonal to $e_x$ in $T_x\cd$ and such  that $\det(e_x, f_x)=1$. 
  Since $\mathrm{Poin}_\Gamma\rest{\widetilde\Gamma}$ is equivalent to the metric  induced by the ambient Riemannian metric,
  there exists a constant $C$ such that for all $x\in \widetilde\Gamma$, 
  $C\inv\leq \norm{e_x}\leq C$. Thus, 
   the basis $(e_x, f_x)$ differs from an orthonormal basis by bounded multiplicative constants, i.e. 
   there exists $C\inv\leq \alpha(x)\leq C$ such that $(\alpha(x)e_x, \alpha\inv(x)f_x)$ is orthonormal. 
 
Let us work  in the frame $\set{(e_x, f_x), x\in   \Gamma}$. Since $df\rest{\Gamma}$ is an isometry for the Poincaré metric    
and $f(\Gamma) = \Gamma$, 
the matrix expression of $df_x$ in this frame is of the form 
$$\begin{pmatrix} 
e^{i\theta(x)} & a(x) \\ 0 & e^{-i\theta(x)} J
\end{pmatrix},$$
where $J$ is the (constant) Jacobian. Fix $\lambda$ such that 
$\abs{J}<\lambda < 1$, and for $\e>0$, let $\mathcal{C}_x^\e\subset T_x\cd$ be the cone defined by 
$$\mathcal{C}_x^\e = \set{ue_x+vf_x, \ \abs{v} \leq  \e \abs{u}}.$$ 
Let also $A = \sup_{x\in \widetilde \Gamma}\abs{a(x)}$. 
Working in coordinates, if  $(u,v)\in  \mathcal{C}_x^\e$ then 
$$df_x(u,v) = :(u_1, v_1) = (e^{i\theta(x)} u + a(x) v , e^{-i\theta(x)}Jv),$$
hence $$ \abs{u_1} \geq \abs{u} - A \abs{v} \geq \abs{u} (1-A\e)
\text{ and } \abs{v_1} = \abs{Jv} \leq \e \abs{J} \abs{u} $$
We see that if $\e$ is so small that $\abs{J} < \lambda(1-A\e)$, then  for every $x\in \widetilde \Gamma$ we have that 
$\abs{v_1}\leq \lambda\e\abs{u_1}$, that is, 
$df_x(\mathcal C _x^\e)\subset \mathcal{C}_{f(x)}^{\la\e}$. The proof is complete.   
 \end{proof}

 \end{document}